\documentclass[11pt, a4paper]{amsart}

\usepackage{amsthm} 				
\usepackage{amsmath}				
\usepackage{amssymb}				
\usepackage{amsfonts}

\usepackage{csquotes}				

\usepackage[T1]{fontenc} 				

\usepackage[a4paper, total={16.5cm, 23.2cm}]{geometry}				

\usepackage{tikz-cd}					
\usepackage[all]{xy}					
\usepackage{pgf,tikz}
\usepackage{tikzsymbols}
\usepackage{pgf,tikz}
\usepackage{graphicx}
\usetikzlibrary{shapes}
\usetikzlibrary{cd}
\usetikzlibrary{arrows}
\usetikzlibrary{topaths}
\usetikzlibrary{decorations.pathmorphing}
\usetikzlibrary{shadows.blur}

\usetikzlibrary{calc}
\usepackage{mathrsfs}
\usepackage{amsfonts}
\usetikzlibrary{arrows}
\usepackage[mathscr]{eucal}
\usepackage{ulem}

\usepackage{graphicx}
\usepackage{caption}
\usepackage{subcaption}
\usepackage{hyperref}

\newcommand{\End}{\mbox{End}}
\newcommand{\Hom}{\mbox{Hom}}

\newcommand{\Fac}{\mbox{Fac}}
\newcommand{\Sub}{\mbox{Sub}}

\renewcommand{\mod}{\mbox{mod}}

\newtheorem{theorem}{Theorem}[section]
\newtheorem{corollary}[theorem]{Corollary}
\newtheorem{lemma}[theorem]{Lemma}

\newtheorem{conjecture}[]{Conjecture}

\usepackage[color=blue!20!white]{todonotes} 

\theoremstyle{definition}
\newtheorem{definition}[theorem]{Definition}

\theoremstyle{remark}

\begin{document}
\normalem

\title{A $\tau$-tilting approach to the first Brauer-Thrall conjecture}

\author{Sibylle Schroll}

\address{School of Mathematics and Actuarial Sciences, University of Leicester, Leicester, LE1  7RH}
\email{schroll@leicester.ac.uk, hjtc1@leicester.ac.uk}
\thanks{The authors are supported by the EPSRC through the Early Career Fellowship, EP/P016294/1.
}

\author{Hipolito Treffinger}


\maketitle

\begin{abstract}
In this paper we study the behaviour of  modules over finite dimensional algebras whose endomorphism algebra is a division ring.
We show that there are finitely many such modules in the module category of an algebra if and only if the length of all such modules is bounded. 
\end{abstract}

\section{Introduction}

In recent years, representation theory has played a central role in many of the latest developments and breakthroughs in mathematics and mathematical physics, such as stability conditions, DT-invariants and homological mirror symmetry. 
This builds on the history of representation theory of algebras which spans almost two centuries, for a historical account see for example \cite{Gustafson1882}. 
The theory transformed in the late sixties with the introduction of homological methods.
This has unified and re-structured representation theory, establishing the tools and techniques which are widely used today.  
Another great impact on representation theory has resulted from the theory of cluster algebras, leading to the introduction of concepts such as cluster categories, mutation and $\tau$-tilting theory.
In this paper we use some of this newly developed theory in the pursuit of one of the most classical problems in representation theory, the first Brauer-Thrall conjecture.

A key observation on which much of representation theory is based is the fact that any representation uniquely factors into a direct sum of indecomposable representations. 
This quickly led to the now well known dichotomy of algebras into algebras of finite and infinite representation type according to whether there are finitely or infinitely many non-isomorphic indecomposable representations.
A crucial role in this development was played by the first and second Brauer-Thrall conjectures. 
Much of the last century's work in representation theory was motivated by the work on these conjectures.
The first conjecture was proved in the late sixties by Roiter \cite{Roiter1968} for finite dimensional algebras over a field using what is now known as the Gabriel-Roiter measure. 
A different proof of this conjecture for Artinian rings using almost split sequences was given by Auslander \cite{Auslander1974} in the early seventies. 
The second Brauer Thrall conjecture was not proved until the mid eighties when Bautista gave a proof of the conjecture for  finite dimensional algebras over a field \cite{Bautista1985}.

The first Brauer-Thrall conjecture states that for a finite dimensional algebra $A$ over a field, if there is a positive integer $n$ such that every indecomposable finite dimensional $A$-module has less than $n$ composition factors, then $A$ is of finite representation type. 

The second Brauer-Thrall conjecture states that if a finite dimensional algebra $A$ over a field is of infinite representation type, then there are infinitely many integers $d$ such that there exist infinitely many isomorphism classes of indecomposable $A$-modules with $d$ composition factors. 

In parallel to the work on the Brauer-Thrall conjectures, classical tilting theory was introduced in 1980 in \cite{BB1980} as an axiomatic approach to the study of reflection functors. 
It quickly evolved as a powerful tool to relate the representation theory of different algebras via derived equivalences \cite{Happel1988, Rickard1989}.
Many subsequent developments in representation theory are based on tilting theory, making extensive use of  Auslander-Reiten theory. 
In \cite{AIR} the authors take this idea further, introducing the new concept of $\tau$-tilting theory by combining ideas from tilting theory with Auslander-Reiten theory and more specifically the Auslander-Reiten translate $\tau$ which gives the theory its name. 
It quickly emerged that $\tau$-tilting theory is a powerful tool relating representation theory with stability conditions, scattering diagrams and mirror symmetry.

In terms of representation theory, $\tau$-tilting theory has given new ways of describing the module category of a finite dimensional algebra by establishing bijections between $\tau$-tilting pairs, functorially finite torsion classes and $2$-term silting complexes.
In particular, the representation theoretic notion of a \emph{brick}, that is a module whose endomorphism algebra is a division ring, plays a central role in $\tau$-tilting theory.  
Namely, in \cite{DIJ19} it is shown that the module category of an algebra has finitely many $\tau$-tilting pairs if and only if it has finitely many bricks. 
These algebras, known as \textit{$\tau$-tilting finite} algebras, play a similar role in $\tau$-tilting theory as representation finite algebras play in classical representation theory. 
For example, each $\tau$-tilting finite algebra has a well-defined category of wide subcategories \cite{Buan2019}, the support of their scattering diagram of a $\tau$-tilting finite algebra is completely determined by their $\tau$-tilting pairs \cite{Bridgeland2017, BST2019} and the stability manifold of a finite dimensional algebra $A$ is contractible if the algebra is silting-discrete which implies, in particular, that the heart of any bounded t-structure of the derived category of $A$ is a module category over a $\tau$-tilting finite algebra \cite{PSZ18}.
This analogy between $\tau$-tilting finite and representation finite algebras can be extended even further, since one can specify the classical Brauer-Thrall conjectures from arbitrary indecomposable representations to bricks, obtaining in this way a $\tau$-tilting version of these conjectures, see \cite{STV19}.

\begin{conjecture}[First $\tau$-Brauer-Thrall conjecture]\label{conj:tBT1}
Let $A$ be a finite dimensional $K$-algebra. 
If there is a positive integer $n$ such that every finite dimensional brick has less than $n$ composition factors, then $A$ is $\tau$-tilting finite.
\end{conjecture}

\begin{conjecture}[Second $\tau$-Brauer-Thrall conjecture]\label{conj:tBT2}
Let $A$ be a $\tau$-tilting infinite finite dimensional $K$-algebra. 
Then there is an integer $d$ such that there exist infinitely many isomorphism classes of bricks with $d$ composition factors.
\end{conjecture}

In this paper we show that Conjecture \ref{conj:tBT1} holds. 

\begin{theorem}\label{thm:tBT1}
Let $A$ be a finite dimensional $K$-algebra. 
If there is a positive integer $n$ such that every finite dimensional brick has less than $n$ composition factors, then $A$ is $\tau$-tilting finite.
\end{theorem}

As a direct consequence we obtain the following new proof of the classical first Brauer-Thrall conjecture for $\tau$-tilting infinite algebras.

\begin{corollary}
Let $A$ be finite dimensional $K$-algebra. 
If $A$ is $\tau$-tilting infinite, then there are indecomposable $A$-modules with arbitrarily many composition factors.
\end{corollary}

It is important to remark, however, that the first $\tau$-Brauer-Thrall conjecture is not equivalent to the first Brauer-Thrall conjecture, since there are representation infinite $\tau$-tilting finite algebras.
Examples of such algebras are preprojective algebras of quivers arising in representation theory \cite{Mizuno2014}, contraction algebras arising in the homological minimal model program \cite{August2018} and most tame blocks of group algebras arising in the representation theory of finite groups~\cite{Erdmann1990, JER}.
 
Conjecture 2 is open in general, but it has been shown to hold for particular classes of algebras.
Namely, in \cite{Mousavand2019-2} Conjecture 2 has been shown for non-distributive and biserial algebras using a  more geometric approach, see also  \cite{MousavandPhD}.
Another proof of Conjecture 2 for special biserial algebras is given in \cite{STV19} in the context of a detailed study of band modules and bricks.
 
We note that if Conjecture 2 holds then it follows that for $\tau$-tilting infinite algebras the second Brauer-Thrall conjecture holds by the results in \cite{Smalo1980}.
\medskip

We prove Theorem~\ref{thm:tBT1} by using two important notions arising from the theory of cluster algebras, namely \textit{$g$-vectors} and \textit{$c$-vectors}. 
In the context of cluster algebras $g$- and $c$-vectors are defined independently from each other. However, they are related by the so-called \textit{tropical dualities} \cite{Derksen2010, Nakanishi2012}.

In representation theory, $g$-vectors have a natural interpretation: the $g$-vector of a module encodes its minimal projective presentation \cite{AIR, DIJ19}. 
In \cite{Fu2017}, $c$-vectors have then been defined for any algebra using the tropical duality equation, even when there is no underlying cluster algebra.

In Theorem~\ref{thm:c-vec} we give a representation theoretic interpretation of $c$-vectors by relating each $c$-vector  to at least one brick in the module category.
Before giving the precise statement of Theorem~\ref{thm:c-vec}, given a finite dimensional algebra $A$, we define the following $n \times n$ matrix 
$$
D_A:=\left(\begin{matrix}
				\delta_{1} & \dots   & 0 \\
				\vdots           & \ddots & \vdots     \\
				0                   & \dots   & \delta_{n}
		\end{matrix} \right)
$$
where $\delta_i$ is the dimension over $K$ of the endomorphism algebra of the simple $A$-module $S(i)$, where $\{S(1), \dots, S(n)\}$ is a complete family of representatives of isomorphism classes of simple $A$-modules.
Note that in the case of some hereditary algebras such as, for example, any non-cyclic orientation of $B_n$, the matrix $D_A$ coincides with the diagonal matrix associated to the corresponding skew-symmetrisable cluster algebra.

The proof of Theorem 1.1 highlights the importance of $c$-vectors in representation theory. They are the main object of study in Theorem~\ref{thm:c-vec} which is our second main result. It generalises  \cite[Theorem 3.5]{Tre19} to  algebras over a field which is not necessarily algebraically closed. Loosely speaking, the result can be seen   as a categorification of $c$-vectors. 
For this denote by $[M]$ the class of an $A$-module $M$ in the Grothendieck group $K_0(A)$. 
Recall that $K_0(A)$ is generated by the classes $[S(1)], \dots, [S(n)]$.

\begin{theorem}\label{thm:c-vec}
Let $A$ be a finite dimensional $K$-algebra. 
Then for every $c$-vector $\mathsf{c}$ of $A$ there exists a brick $B_\mathsf{c}$ in the module category of $A$ and a non-zero integer $m_\mathsf{c}$ such that  
$$D_A[B_\mathsf{c}]=m_\mathsf{c}\mathsf{c}.$$
\end{theorem}

 {\bf Acknowledgements:} The authors would like to thank Ibrahim Assem and Yadira Valdivieso D\'iaz for many stimulating conversations and exchanges.

\section{Proof of Theorem~\ref{thm:tBT1}}

\subsection{Setting}
In this paper $A$ is a  finite dimensional $K$-algebra. 
Since any $K$-algebra is Morita equivalent to a basic algebra, without loss of generality we can assume that $A$ is basic. 
We denote by $\mod A$ the category of finitely generated right $A$-modules.
As a right $A$-module, $A$ decomposes into a direct sum of non-isomorphic indecomposable right  $A$-modules $A = P(1) \oplus \ldots \oplus P(n)$. 
Every projective $A$-module is a direct sum of copies of indecomposable projective $A$-modules $P(i)$ and we define $S(i)$ to be the simple $A$-module with projective cover $P(i)$, for $1 \leq i \leq n$.
 
In order to define the Auslander-Reiten translation $\tau$ for a finite dimensional algebra $A$, we recall that the Nakayama functor $\nu$ is the composition of the dualities $\Hom_K(-,K)$ followed by $ \Hom_A(-,A)$.  
Recall that for a projective $A$-module $P$ we have that $\nu P$ is an injective $A$-module and $\Hom_A(P, X)=0$ if and only if $\Hom_A(X, \nu P)=0$ for all $A$-modules $X$.
For an $A$-module $M$, $\tau M$ is defined as the kernel of the image under $\nu$ of the minimal projective presentation of $M$. 
The Auslander-Reiten translation $\tau$ induces a bijection between the set of non-isomorphic non-projective indecomposable $A$-modules and  the set of non-isomorphic non-injective indecomposable $A$-modules. 

The Grothendieck group $K_0(A)$ of $A$ is the free abelian group generated by the indecomposable $A$-modules subject to the relations induced by short exact sequences.  
It is a consequence of the Jordan-H\"older Theorem that the isomorphism classes of simples $A$-modules form a basis of $K_0(A)$. 
Thus if there are $n$ isomorphism classes of simple $A$-modules then  $K_{0}(A)$ is isomorphic to $\mathbb{Z}^{n}$. 
Denote by $\Delta: K_{0}(A)\to \mathbb{R}^{n}$ the embedding of $K_{0}(A)$ into $\mathbb{R}^{n}$ given by $\Delta(M)=[M]$ where $[M]$ is the class of $M$ in $K_0(A)$, for every $A$-module $M$. 
Using this identification, we say that $[M]$ is a vector of $\mathbb{R}^{n}$.

\subsection{$\tau$-tilting theory}
We recall the tools from $\tau$-tilting theory needed for the proof of Theorem~\ref{thm:tBT1}. 
The central notion in $\tau$-tilting theory is the notion of $\tau$-tilting pairs defined in \cite{AIR}.

\begin{definition}
Let $A$ be a finite dimensional $K$-algebra, $M$ an $A$-module and $P$ a projective $A$-module. 
The module $M$ is called $\tau$-rigid if $\Hom_{A}(M,\tau M)=0$ and the pair $(M,P)$ is called \textit{$\tau$-rigid} if $M$ is $\tau$-rigid and $\Hom_{A}(P,M)=0$.
Moreover, we say that $(M,P)$ is \textit{$\tau$-tilting} if the  number of non-isomorphic indecomposable direct summands of $M\oplus P$ is equal to $n$ where $n$ is the rank of the Grothendieck group $K_0(A)$ of $A$.
\end{definition}

When we consider $\tau$-rigid pairs $(M,P)$, we implicitly assume that $M$ and $P$ are basic, that is if $M=\bigoplus_{i=1}^kM_i$ and $P=\bigoplus_{j=k+1}^tP_j$ then $M_i $ is not isomorphic to $ M_j$ and $P_i$ is not isomorphic to $ P_j$ for $i \neq j$. 
Furthermore, it is shown in \cite{AIR} that for two distinct $\tau$-tilting pairs $(M,P)$ and $(M', P')$, the modules  $M$ and $ M'$ are not isomorphic.  

\subsection{Integer vectors associated to modules}

The notions of $c$-vectors and $g$-vectors first arose in the context of cluster algebras (see \cite{FZ4}). 
In this paper we will use the representation theoretic versions of these vectors.

\begin{definition}
Let $A$ be a finite dimensional $K$-algebra, $M$ an $A$-module and $P_1 \to P_0 \to M \to 0$
the minimal projective presentation of $M$, with $P_0=\bigoplus\limits_{i=1}^n P(i)^{a_i}$ and $P_1=\bigoplus\limits_{i=1}^n P(i)^{a'_i}$. 
The integer vector $g^M=(a_1-a'_1, a_2-a'_2,\dots, a_n-a'_n)$ is called the \emph{$g$-vector of $M$}.
\end{definition}

The proof of Theorem~\ref{thm:tBT1} is based on the following properties of $g$-vectors. 
By \cite{AIR, DIJ19} if $(M,P)$ is a $\tau$-tilting pair then the set 
$\{g^{M_{1}},\dots, g^{M_{k}},-g^{P_{k+1}},\dots, -g^{P_{n}}\}$ is a basis of $\mathbb{Z}^{n}$ and if $M$ and $N$ are two $\tau$-rigid modules such that $g^M=g^N$, then $M$ and $N$ are isomorphic.

Another property of $g$-vectors we need is based on a general result of Auslander and Reiten in \cite{AR1985} which holds for any Artinian algebra which is finitely generated over its centre. 
In order to state this property, given a finite dimensional algebra $A$, we need to define the following inner product $\langle - ,- \rangle_A$ on $\mathbb{R}^n$
$$\langle v,w\rangle_A=(v_1, \dots, v_n)D_A\left(\begin{matrix} w_1\\ \vdots \\ w_n \end{matrix}\right)
=\sum_{i=1}^n \delta_i v_i w_i$$
where $D_A$ is defined as in the introduction.

\begin{lemma}\label{lem:ARequation}
Let $A$ be a finite dimensional $K$-algebra and let $M$ and $N$ be two $A$-modules.
Then $$\langle g^{M},[N]\rangle_A=\dim_K(\Hom_A(M,N))-\dim_K(\Hom_A(N,\tau_A M)).$$
\end{lemma}

\begin{proof}
Let $P_1 \to P_0 \to M \to 0$ be the minimal projective presentation of $M$.
Then, by \cite{AR1985}, we have that 
$$\dim_K(\Hom_A(M,N)) - \dim_K(\Hom_A(N,\tau_A M)) = \dim_K(\Hom_A(P_0,N)) - \dim_K(\Hom_A(P_1,N)),$$
where $P_0=\bigoplus\limits_{i=1}^n P(i)^{a_i}$ and $P_1=\bigoplus\limits_{i=1}^n P(i)^{a'_i}$. 
Since the $\Hom$ functor is additive we have 
$$\dim_K(\Hom_A(P_0,N)) - \dim_K(\Hom_A(P_1,N)) =\sum_{i=1}^n (a_i-a'_i) \dim_K(\Hom_A(P(i),N)).$$
Thus it is enough to show $\dim_K(\Hom_A(P(i), N)) = \langle g^{P(i)}, [N]\rangle_A$.
But $ \dim_K(\Hom_A(P(i), N)) = \dim_K(\End_A(S(i)))[N]_i$, where $[N]_i$ is the $i$-th component of $[N]$ and the result follows.
\end{proof}

In \cite{Fu2017} Fu defines $c$-vectors for algebras over an algebraically closed field, but this definition still holds for arbitrary fields.

\begin{definition}\label{def:C-mat}
Let $(M,P)$ be a $\tau$-tilting pair and let \sloppy $\{g^{M_{1}},\dots, g^{M_{k}},-g^{P_{k+1}},\dots, -g^{P_{n}}\}$ be the corresponding basis of $g$-vectors of $\mathbb{Z}^n$.
Define the $g$-matrix $G_{(M,P)}$ of $(M,P)$ as 
$$G_{(M,P)}=\left(g^{M_{1}} ,\dots, g^{M_{k}},-g^{P_{k+1}},\dots , -g^{P_{n}}\right)=\left( \begin{matrix}  
				(g^{M_{1}})_{1} & \dots   & (-g^{P_{n}})_{1} \\
				\vdots      & \ddots & \vdots     \\
				(g^{M_{1}})_{n}  & \dots   & (-g^{P_{n}})_{n}
		\end{matrix} \right).$$
Then the $c$-matrix $C_{(M,P)}$ of $(M,P)$ is defined as
$C_{(M,P)}=(G^{-1}_{(M,P)})^{T}.$ 
Each column of $C_{(M,P)}$ is called a \textit{$c$-vector} of $A$. 
\end{definition}

We now show Theorem~\ref{thm:c-vec}. 
For that recall the following definitions.
For an $A$-module  $M$, we denote by $M^\perp$ the class of $A$-modules $X$ such that $\Hom_A(M,X)=0$.
Dually, we denote by $^\perp M$ the class of $A$-modules $Y$ such that $\Hom_A(Y,M)=0$.
Furthermore, for an $A$-module $M$ we define $\Fac M$ (respectively, $\Sub M$) as the class of all quotients (respectively, subobjects) of an object in the full additive category generated by $M$.

\begin{proof}[Proof of Theorem \ref{thm:c-vec}]
Let $(M,P)$ be a $\tau$-tilting pair and let $\{g^{1},\dots,g^{n}\}$ be the associated basis of $g$-vectors of  $\mathbb{Z}^n$.
Then, from $(M,P)$ we can construct $n$ different $\tau$-rigid pairs $(\hat{M}_1, \hat{P}_1), \dots, (\hat{M}_n, \hat{P}_n)$, where $(\hat{M}_i, \hat{P}_i)$ is obtained from $(M,P)$ by removing the $i$-th indecomposable direct summand of $M$, if $1 \leq i \leq k$, or the $i$-th indecomposable direct summand of $P$, if $k+1 \leq i \leq n$.
It follows from \cite[Theorem~3.8]{Jasso2015} that $\hat{M}_i^{\perp}\cap ^\perp\tau\hat{M}_i \cap \hat{P}_i^\perp$ is equivalent to the module category of a local algebra for all $1 \leq i \leq n$.
Thus, for every $1 \leq i \leq n$, there exists a unique brick $B_i$ in $\hat{M}_i^{\perp}\cap ^\perp\tau\hat{M}_i \cap \hat{P}_i^\perp$.
Consider the matrix
$$X_{(M,P)}:=\left(\begin{matrix}
				[B_{1}] ,& [B_2] ,& \dots   & [B_{n}]
		\end{matrix}\right)=\left(\begin{matrix}
				[B_{1}]_{1} & \dots   & [B_{n}]_{1} \\
				\vdots      & \ddots & \vdots       \\
				[B_{1}]_{n}  & \dots   & [B_{n}]_{n}
		\end{matrix}\right)
$$
having as the $i$-th column the vector $[B_i]$ corresponding to the brick $B_i$.
Now, we multiply $G_{(M,P)}^{T}$, the transpose of the $G$-matrix of $(M,P)$, by the matrix $D_A$ and $X_{(M,P)}$ in that order
$$
G_{(M,P)}^{T} D_{A} X_{(M,P)}= \left(\begin{matrix}
				\langle g^1, [B_1]\rangle_A & \langle g^1, [B_2]\rangle_A & \dots   & \langle g^1, [B_n]\rangle_A \\
				\langle g^2, [B_1]\rangle_A & \langle g^2, [B_2]\rangle_A & \dots   & \langle g^1, [B_n]\rangle_A \\
				\vdots     &  \vdots    & \ddots & \vdots     \\
				\langle g^n, [B_1]\rangle_A  & \langle g^n, [B_2]\rangle_A   & \dots   & \langle g^n, [B_n]\rangle_A
		\end{matrix} \right).
$$
Note that for  $j \neq i$,  $\hat{M}_i^{\perp}\cap ^\perp\tau\hat{M}_i \cap \hat{P}_i^\perp$ is a subcategory of $M_j^\perp \cap ^\perp \tau M_j$ if $1 \leq j \leq k$, and $\hat{M}_i^{\perp}\cap ^\perp\tau\hat{M}_i \cap \hat{P}_i^\perp$ is a subcategory of $P_j^\perp$ if $k+1 \leq j \leq n$ .
In particular, $B_i$ is in $M_j^\perp\cap ^\perp\tau M_j$ if $1 \leq j \leq k$ or $B_i$ is in $P_j^\perp$ if $k+1 \leq j \leq n$.
Then, by Lemma~\ref{lem:ARequation} we have $\langle g^{j}, [B_i] \rangle_A = 0$, for all for $1 \leq i,j \leq n$ and  $j \neq i$. 
This implies that $G_{(M,P)}^{T} D_{A} X_{(M,P)}$ is equal to a diagonal matrix $D$.
Thus $D_A X_{(M,P)}=C_{(M,P)}D$.
As a consequence $D_A[B_i]=\langle g^i, [B_i]\rangle_A\mathsf{c}_i$
where $\mathsf{c}_i$ is the $i$-th column of $C_{(M,P)}$.
Now, $D_A[B_i]=(\delta_1[B_i]_1, \dots, \delta_n[B_i]_n)$ where $\delta_k$ is non-zero for all $1 \leq k \leq n$ and $[B_i]_l$ is non-zero for some $1 \leq l \leq n$ because $B_i$ is a non-zero $A$-module. 
Then $D_A[B_i]=\langle g^i, [B_i]\rangle_A\mathsf{c}_i$ implies immediately that $\langle g^i, [B_i]\rangle_A$ is a non-zero integer and the result follows.
\end{proof}

\begin{corollary}\label{cor:sign-coherence}
Let $A$ be a $K$-algebra.
Then every $c$-vector of $A$ is sign-coherent, that is, all entries of a $c$-vector are either non-positive or non-negative.
\end{corollary}

We note that Corollary~\ref{cor:sign-coherence} was shown in \cite{Fu2017} for algebras over an algebraically closed field.

\begin{proof}[Proof of Theorem~\ref{thm:tBT1}]
Suppose that $A$ is $\tau$-tilting infinite and let $t$ be a positive integer. 
Then $A$ has infinitely many $\tau$-tilting pairs and hence infinitely many distinct $g$-vectors giving rise to infinitely many distinct $G$-matrices. 
As a consequence, there are infinitely many different $C$-matrices for $A$.
Thus, $A$ has infinitely many different $c$-vectors.
Let $\delta_A = \max\{\delta_i : 1 \leq i \leq n \}$, where, as before, $\delta_i = \dim_K(End_A(S(i)))$.
Therefore, using the sign-coherence of $c$-vectors,  by the pigeon hole principle, for every positive integer $t$ there is a $c$-vector $\mathsf{c}_t= (c_1, \dots, c_n)$ such that $|\sum_{i=1}^n c_i| = \sum_{i=1}^n |c_i| \geq \delta_At$.
Then Theorem~\ref{thm:c-vec} implies the existence of an $A$-module $B_{t}$ which is a brick and an integer $m_{t}$ such that $D_A[B_t]=m_t\mathsf{c}_t$.
Hence, $\sum_{i=1}^n \delta_A[B_{t}]_i \geq \sum_{i=1}^n \delta_i[B_{t}]_i =  \sum_{i=1}^n |m_{t}c_i| =|m_t|\sum_{i=1}^n |c_i| \geq |m_t|\delta_A t \geq \delta_A t$.
Hence for every positive integer $t$, there exists a brick $B_t$ such that the number of composition factors of $B_t$ is greater or equal to $t$.
\end{proof}

\def\cprime{$'$} \def\cprime{$'$}

\end{document}